\documentclass[12pt, twoside]{article}
\usepackage{latexsym,enumerate}
\usepackage{amsmath}
\usepackage{amsthm,amsopn}
\usepackage{amstext,amscd,amsfonts,amssymb,fontenc,bbm}
\usepackage[ansinew]{inputenc}
\usepackage{verbatim}
\usepackage{mathtools}
\usepackage{epsfig}

\usepackage{lineno} %?
\usepackage{textcomp}

\newtheorem{theorem}{Theorem}[section]
\newtheorem{corollary}[theorem]{Corollary}

\newtheorem{conjecture}[theorem]{Conjecture}
\newtheorem{proposition}[theorem]{Proposition}

\newtheorem{lemma}[theorem]{Lemma}
\newtheorem{observation}[theorem]{Observation}

\newtheorem{remark}[theorem]{Remark}
\theoremstyle{definition}

\DeclarePairedDelimiter\floor{\lfloor}{\rfloor}
\def\AA{{\mathcal A}}
\def\BB{{\mathcal B}} 
\def\RR{{\mathcal R}}
\def\SS{{\mathcal S}}
\def\II{{\mathcal I}}

\begin{document}
%\linenumbers

\title{Independence and matching numbers of some token graphs}

\author{
Hern\'an de Alba\thanks{Catedr\'atico CONACYT-UAZ, Mexico.
Email: heralbac@gmail.com, waltercarb@gmail.com.}
\and
Walter Carballosa\thanks{Department of Mathematics and Statistics, Florida International University, 11200 S.W. 8th Street, Miami, Florida 33199. U.S.A. Email: waltercarb@gmail.com.}
\and
Jes\'us Lea\~nos\thanks{Unidad Acad\'emica de Matem\'aticas, Universidad Aut\'onoma de Zacatecas, Calzada Solidaridad y Paseo La Bufa, Col.
Hidr\'aulica, C. P. 98060, Zacatecas, Zac., Mexico.
Email: jesus.leanos@gmail.com, luismanuel.rivera@gmail.com.} 
\and
Luis Manuel Rivera\footnotemark[3]
}
\date{}

\maketitle{}

\begin{abstract}

Let $G$ be a graph of order $n$ and let $k\in\{1,\ldots,n-1\}$. The $k$-token graph $F_k(G)$ of $G$, is the graph whose vertices are the $k$-subsets of $V(G)$, where two vertices are adjacent in $F_k(G)$ whenever their symmetric difference is an edge of $G$. We study the independence and matching numbers of $F_k(G)$. We present a tight lower bound for the  matching number of $F_k(G)$ for the case in which $G$ has either a perfect matching or an almost perfect matching. Also, we estimate the independence number for bipartite $k$-token graphs, and determine the exact value for some graphs. 

\end{abstract}

{\it Keywords:}  Token graphs; Matchings; Independence number; Johnson graphs.\\
{\it AMS Subject Classification Numbers:}    05C10; 05C69.

\section{Introduction}
 
Throughout this paper, $G$ is a simple finite graph of order $n\geq 2$ and $k\in \{1,\ldots, n-1\}$. The 
$k$-\textit{token graph} $F_k(G)$ of $G$ is the graph whose vertices are all the $k$-subsets of $V(G)$ and two $k$-subsets are adjacent if their symmetric difference is an edge of $G$. In particular, observe that $F_1(G)$ and $G$ are isomorphic, which, as usual, is denoted by $ F_1(G)\simeq G$. Moreover, note also that $F_k(G)\simeq F_{n-k}(G)$.  Often, we simply write token graph instead of $k$-token graph.

\subsection{Token graphs}
The origins of the notion of token graphs can be dated back to the 90's in the works of Alavi et al. \cite{alavi2, alavi3}, where the $2$-token graphs are called double vertex graphs. Later, Terry Rudolph \cite{rudolph} used $F_k(G)$ to study the graph isomorphism problem.
In such a work Rudolph gave examples of non-isomorphic graphs $G$ and $H$ which are cospectral, but with $F_2(G)$ and $F_2(H)$ non-cospectral. He emphasized that fact by making the following remark about the eigenvalues of $F_2(G)$:  ``{\em then the eigenvalues of this larger matrix are a graph invariant, and in fact are a more powerful invariant than those of the original matrix $G$}". 

In 2007 \cite{aude}, the notion of token graphs was extended by Audenaert et al. to any integer $k\in\{1,\ldots, n-1\}$ where $F_k(G)$ is called the {\em symmetric $k$-$th$ power of $G$}. It   was proved in \cite{aude} that the $2$-token graphs of strongly regular graphs with the same parameters are cospectral. In addition, some connections with generic exchange Hamiltonians in quantum mechanics were also discussed. Following Rudolph's study,  Barghi and Ponomarenko \cite{barghi} and Alzaga et al. \cite{alzaga} proved, independently, that for a given positive integer $m$ there exists infinitely many pairs of non-isomorphic graphs with cospectral $m$-token graphs.

In 2012 Ruy Fabila-Monroy et al. \cite{FFHH} reintroduced the concept of $k$-token graphs as ``{\em a model in which $k$ indistinguishable tokens move from vertex to vertex along the edges of a graph}" and began the systematic study of the combinatorial parameters of $F_k(G)$. In particular, the investigation presented in \cite{FFHH} includes the study of connectivity, diameter, cliques, chromatic number, Hamiltonian paths, and Cartesian products of token graphs. This line of research was continued for several authors (see, e.g., \cite{alba,token2, gomez, leatrujillo}).          

From the model of $F_k(G)$ proposed in \cite{FFHH} it is clear that the $k$-token graphs can be considered as part of several models of swapping in the literature  \cite{vanden,yama} that are part of reconfiguration problems (see e.g. \cite{caline,ratner}). 

As an example of the relationship between reconfiguration problems and problems involving the determination of parameters of $F_k(G)$, let us consider the problem of determining $diam(F_k(G))$ the diameter of $F_k(G)$ and the pebble motion (PM) problem. We recall that the PM problem asks if an arrangement $A$ of $1< k < |G|$ distinct pebbles numbered
 $1,\ldots, k$ and placed on $k$ distinct vertices of $G$ can be transformed into another given arrangement $B$ by moving the pebbles along edges of $G$ provided that at any given time at most one pebble is traveling along an edge and each vertex of $G$ contains at most one pebble. The PM problem has been studied in 
 ~\cite{auletta} and~\cite{miller} from the algorithmic point of view. Also, in such papers several applications of the PM problem have been mentioned, which include the management of memory in totally distributed computing systems and problems in robot motion planning. On the other hand, note that the PM problem is a variant of the problem of determining $diam(F_k(G))$ (the only difference is that in the  PM problem, the pebbles or tokens are distiguishable).  

The $k$-token graphs also are a generalization of Johnson graphs: if $G$ is the complete graph of order $n$, then $F_k(G)$ is isomorphic to the 
Johnson graph $J(n,k)$. The Johnson graphs have been studied from several approaches, see for example~\cite{alavi,riyono,terwilliger}. 
In particular, the determination of the exact value of the independence number 
$\alpha(J(n,k))$ of the Johnson graph, as far as we know, remains open in its generality, albeit it has been widely studied~\cite{brou,brou2,etzion,Jo,mira}. Possibly, the last effort to determine $\alpha(J(n,k))$ was made by
 K. G. Mirajkar et al. in 2016~\cite{mira}. In such a work they presented an exact formula for $\alpha(J(n,k))$, which is unfortunately wrong: 
the independence number of $J(7, 3)$ is $7$ because it is equal to the distance-4 constant weight code $A(7, 4, 3)$~\cite{brou}, but the formula in~\cite{mira} gives $6$.
   
\subsection{Main results}

The graph parameters of interest in this paper are the independence number 
and the matching number. A set $I$ of vertices in a graph $G$ is an \emph{independent set} if no two vertices in $I$ are adjacent; a \emph{maximal independent set} is an independent set such that it is not a proper subset of any independent set in $G$. The \emph{independence number} $\alpha(G)$ of  $G$ is the number of vertices in a largest independent set in $G$, and its computation is NP-hard~\cite{karp}.

A {\it matching} in a graph $G$ is a subset $M$ of edges of $G$ such that no two edges in $M$ have a vertex in common. In this paper 
we will use $||M||$ to denote the number of edges in the matching $M$. The {\it order} $|M|$ of a matching $M$ is the number of vertices involved in the edges of $M$, that is $|M|=2||M||$.  The {\it matching number} $\nu(G)$ of $G$ is the number of edges of any largest matching in $G$. A matching $M$ of $G$ is called {\it perfect matching} (respectively {\it almost perfect matching}) if the number of vertices (order) $|M|$ of $M$ is equal to
  $|G|$ the number of vertices of $G$ (respectively, if $|M|=|G|-1$). Note that $\nu(G)=|G|/2$ if and only if $G$ has a perfect matching. Similarly, $\nu(G)=(|G|-1)/2$ if and only if $G$ has an almost perfect matching. In this case, Jack Edmonds proved in 1965 that the matching number of a graph can be determined in polynomial time \cite{edmonds}.   

Our main results in this paper are Theorems \ref{th:match}, \ref{th:Kmn} and \ref{k2cicle}. In our attempt to estimate the independence number of the token graphs of certain families of bipartite graphs, we meet the following natural question:
  
{\bf Question 1.} If $\nu(G)= \lfloor |G| /2\rfloor$,  what can we say about $\nu(F_k(G))$?

In Theorem \ref{th:match} we answer Question 1 by providing a lower bound for 
$\nu(F_k(G))$ and exhibiting some graphs for which such a bound is tight. 
\begin{theorem}\label{th:match}
Let $G$ be a graph of order $n$ and let $k$ be an integer with $1\le k\le n-1$. If
$\nu(G)= \lfloor n /2\rfloor$, then 

\begin{itemize}
     
\item[(1)] $\nu(F_k(G))= \binom{n}{k}/2$, if $n$ is even and $k$ is odd.
\item[(2)] $\nu(F_k(G))\geq \big(\binom{n}{k}-\binom{\lfloor n/2\rfloor}{\lfloor k/2\rfloor} \big)/2$, otherwise.

\end{itemize}
Moreover, when $G$ is precisely a perfect matching or an almost perfect matching, the bound (2) is tight.   
\end{theorem}

The proof of Theorem  \ref{th:match} is given in Section \ref{match}. Sections \ref{indep-bipar}, \ref{sec:Kmn}, and \ref{sec:C_p} are mostly devoted to the determination of the exact value of the independence numbers of the token graphs of certain common families of graphs. Our main results in this direction are 
the following results.

\begin{theorem}\label{th:Kmn}  
$\alpha(F_2(K_{m,n}))=\max \{mn, \binom{m+n}{2}-mn\}.$
\end{theorem}

In Section \ref{indep-bipar} we present some results which will be used in the proof of Theorem \ref{th:Kmn} and also help to determine some of the exact values of $\alpha(F_k(G))$ for  $G\in \{P_{m}, C_{2m}, K_{1, m}, K_{m,m},\-K_{m, m+1}\}$ and $2\leq k \leq |G|-2$. The proof of Theorem \ref{th:Kmn} is given 
in Section \ref{sec:Kmn}.

\begin{theorem}\label{k2cicle} If $p$ is a nonnegative integer and $C_p$ is the cycle of length $p$, then $\alpha\left(F_2(C_p)\right)=\floor{p\floor{p/2}/2}.$
\end{theorem}

This formula for $\alpha(F_2(C_p))_{p\geq 3}$ produces the sequence A189889 in OEIS \cite{oeis}, which counts the maximum number of 
non-attacking kings on an $p \times p$ toroidal board (see, e.g., \cite[Theorem 11.1, p. 194]{wat}). In Section \ref{sec:C_p} we show Theorem \ref{k2cicle}.  

%%%%%%%%%%%%%%%%%%%%%%%%%%%%
%%%%%%%%%%%%%%%%%%%%%%%%%%%%
%%%%%%%%%%%%%%%%%%%%%%%%%

\section{Proof of Theorem \ref{th:match}}\label{match}

Since $\nu(G)= \lfloor n /2\rfloor$, then $G$ has a matching $M$ with edges $a_1b_1, a_2b_2, \dots, a_mb_m$, where $m=\lfloor n/2 \rfloor$.
We analyze the two cases separately.\\
{\it Proof of (1).}  Let $X=\{x_1, \dots, x_k\}$ be a vertex of $F_k(G)$. Since $n$ is even and $k$ is odd, then there exists at least one subscript $j\in \{1,\ldots ,k\}$ such that $X$ contains precisely one of $a_j$ or $b_j$. Let $i$ be the smallest of such subscripts, and  let  
$X'=X \Delta \{a_i, b_i\}$. Then $X$ and $X'$ are adjacent in $F_k(G)$. Moreover, because the way in which $X'$ was obtained from $X$, it is
not hard to see that  the set of edges $\{[X, X']~:~X \in V(F_k(G))\}$ is a matching in $F_k(G)$ with exactly $\binom{n}{k}/2$ edges, as required.\\

{\it Proof of (2).} Let $q=\lfloor k/2 \rfloor$. We will show that $F_k(G)$ always contains an independent vertex set $I$ with $|I|=\binom{m}{q}$, and such that the subgraph of  $F_k(G)$ that results by deleting the vertices of $I$ has a perfect matching. This clearly implies the required inequality.    

If $k$ is even, then the set $I_0$ of vertices in $F_k(G)$ of the form $\{a_{j_1},b_{j_1}, \dots a_{j_q}, b_{j_q}\}\subseteq V(M)$ is an independent set in  $F_k(G)$ with exactly $\binom{m}{q}$ vertices. Similarly, if $k$ is odd (and hence $n$ odd due to previous case), then the set $I_1$ of vertices in $F_k(G)$ of the form $\{a_{j_1},b_{j_1}, \dots a_{j_q}, b_{j_q}, v\}\subseteq V(M) \cup \{v\}$, where $v$ is the vertex in $V(G)\setminus V(M)$, is an independent set with the desired number of vertices. Let $I=I_1$ if $k$ is odd, and let  $I=I_0$ otherwise. By applying an analogous procedure to the one used in the proof of part (1) to the vertices of $F_k(G)-I$ we can get the required perfect matching of $F_k(G)-I$. 

Finally, if $G$ is a perfect matching (resp., almost perfect matching) then $I_0$ (resp., $I_1$) is a set of isolated vertices and the final part of the theorem follows.   
$\square$ 

The converse of Theorem~\ref{th:match} (1) is false in general. For example, $F_3(K_{1, 5})$ has a perfect matching (the set of red edges in  Figure~\ref{f2star13}) but $K_{1, 5}$ does not have it.  

\begin{figure}[ht] 
\begin{center}
\epsfig{file=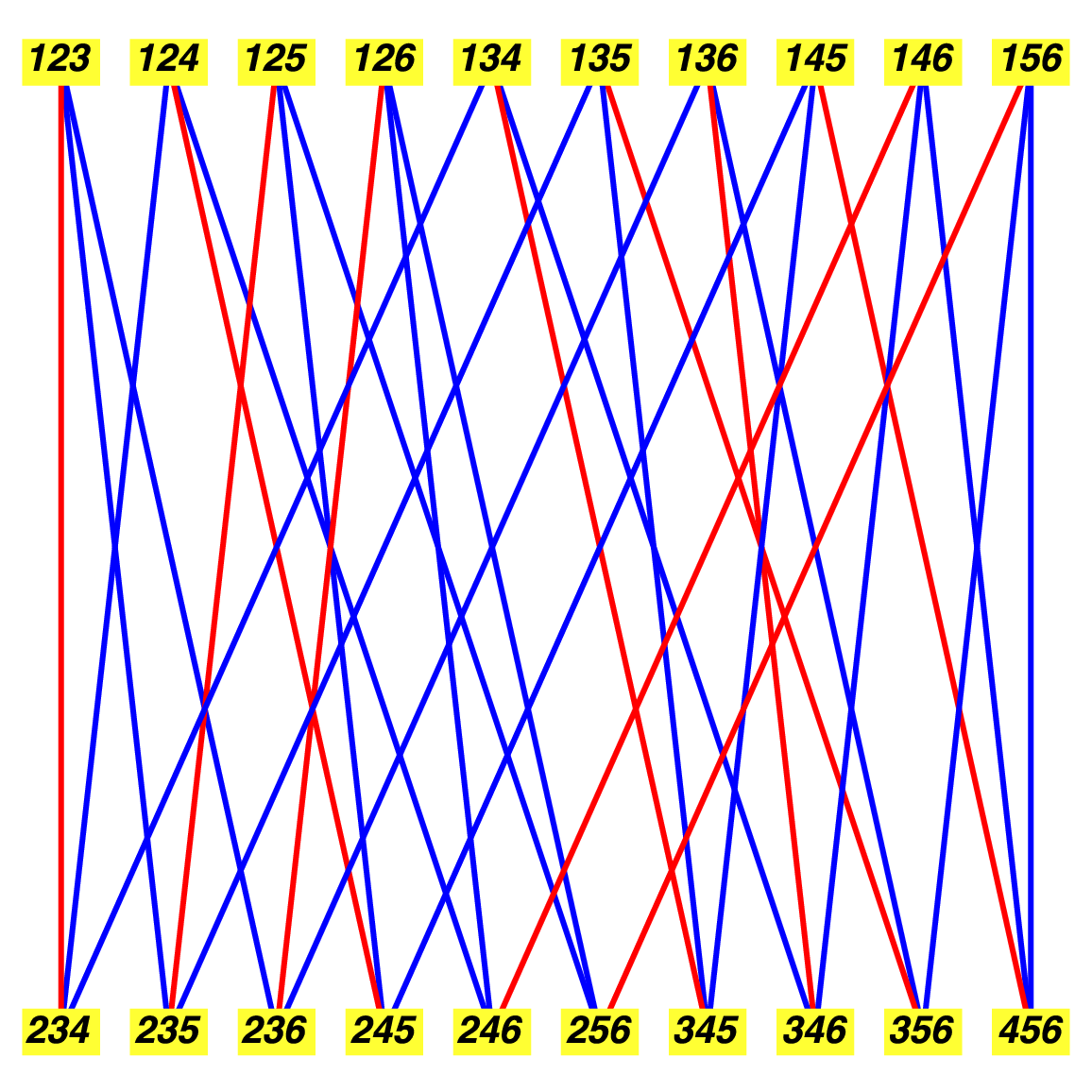, scale=0.3} 
\caption{\small A perfect matching (the red edges) in the $3$-token graph of $K_{1, 5}$.}
\label{f2star13}
\end{center}
\end{figure}

Next corollary states that $2\nu(F_k(G)) \rightarrow |V(F_k(G))|$ when $|G|\rightarrow \infty$.

\begin{corollary}\label{cor:match-even}
Let $G$ be a graph of order $n$. If $\nu(G)= \lfloor n /2\rfloor$, then 
\[
\frac{2\nu(F_k(G))}{\binom{n}{k}}\ge \frac{\binom{n}{k}-\binom{\lfloor n/2\rfloor}{\lfloor k/2\rfloor}}{\binom{n}{k}}\ge 1 -\left(\frac{k}{n}\right)^{ \lfloor k/2 \rfloor}.
\] 
\end{corollary}

%%%%%%%%%%%%%%%%%%%%%%%%%%%%
%%%%%%%%%%%%%%%%%%%%%%%%%%%%
%%%%%%%%%%%%%%%%%%%%%%%%%%%%

\section{Estimation of $\alpha(F_k(G))$ for $G$ bipartite}\label{indep-bipar}

This section is devoted to the study of the independence number of the $k$-token graphs of some common bipartite graphs. As we will see, most of the results stated 
in this section will be used, directly or indirectly, in the proof of Theorem \ref{th:Kmn}. 

%%%%%%%%%%%%%%%%%%%%%%%%%%%%
%%%%%%%%%%%%%%%%%%%%%%%%%%%%
%%%%%%%%%%%%%%%%%%%%%%%%%%%%
%%%%%%%%%%%%%%%%%%%%%%%%%%%%
%%%%%%%%%%%%%%%%%%%%%%%%%%%%
%%%%%%%%%%%%%%%%%%%%%%%%%%%%

\subsection{Notation and auxiliary results}\label{sec:notation}

Let $G$ be a bipartite graph with bipartition $\{B,R\}$. Let $m:=|B|\geq 1, n:=|R|\geq 1$, and let  $k\in \{1,\ldots ,m+n-1\}$. Let  

$$\mathcal R:=\{A\subset V(G) \colon  |A|=k, |R \cap A| \text{ is odd}\},$$
 and let 
$$\mathcal B:= \{A\subset V(G) \colon  |A|=k, |R \cap A| \text{ is even}\}.$$

From Proposition 12 in \cite{FFHH} we know that $F_k(G)$ is a bipartite graph. It is not difficult to check that $\{\mathcal R, \mathcal B\}$ is a bipartition of $F_k(G)$. Without loss of  generality we can assume that $m\leq n$. 

\begin{remark}\label{rem:notation}
Unless otherwise stated, from now on we will assume that 
$G, B, R, \mathcal B,$ $\mathcal R$, $m, n$ and $k$ are as above. 
\end{remark}

Recall that a {\it matching of $B$ into $R$} is a matching $M$ in $G$ such that every vertex in $B$ is incident with an edge in $M$ \cite{asra}. Now we recall the classical Hall's Theorem.

\begin{theorem}\label{hall}
The bipartite graph $G$ has a matching of $B$ into $R$ if and only if $|N(S)|\geq |S|$ for every $S \subseteq B$.
\end{theorem}

\begin{lemma}\label{equal-part}
If there exists a matching of $B$ into $R$, then $\alpha(G)=|R|$.
\end{lemma}

\begin{proof}
Since $G$ contains a matching $M$ of $B$ into $R$, it follows that $|R|\geq |B|$. Then $\alpha(G)\geq |R|$, because $R$ is an independent set of $G$. 

Now we show that $\alpha(G)\leq |R|$.  Let $X$ be any independent set of $G$. If $X \subseteq B$ or $X \subseteq R$ we are done. So we may assume  that $B':=X\cap B\not=\emptyset$ and  that $R':=X\cap R\not=\emptyset$. Let $M'$ be the set of edges in $M$  that have one endvertex in $B'$, and let $R''$ be the set of endvertices of $M'$ in $R$.  
Thus $V(M')=B'\cup R''$, and $|B'|=|R''|$. Since $X$ is an independent set, then $R'\cap R''=\emptyset$, and hence $R'\cup R''$ is also an independent set of $G$ with  
$|R'\cup R''|=|X|.$
\end{proof}

\begin{proposition}\label{prop:bipartite}
Let $G, B, R, \mathcal B, \mathcal R, m, n$ and $k$ as in Remark \ref{rem:notation}. Then     
\[
 \alpha\big(F_k(G)\big) \ge \max\left\{r, \binom{n+m}{k}-r \right\},
\]
where 
\[
r=\sum_{i=1}^{\lceil k/2\rceil} \binom{n}{2i-1}\binom{m}{k-2i+1}.
\]

\end{proposition}

\begin{proof}
For $i=1,\ldots,{\lceil k/2\rceil}$, let $\RR_i$ be the subset of $\RR$ defined by $$\RR_i:=\{A\subset V(G) \colon  |A|=k, |R \cap A|=2i-1\}.$$ Since $|\RR_i|=\binom{n}{2i-1}\binom{m}{k-2i+1}$ the desired result it follows by observing that
 $|\mathcal R|=r$ and $|\mathcal B|= \binom{n+m}{k}-r$.
\end{proof}

The bound for $\alpha(F_k(G))$ given in Proposition~\ref{prop:bipartite} is not always attained: for instance, it is not difficult to see that the graph $G$ in Figure~\ref{g-cont} has $\alpha(F_2(G))=12$ and 
$\max\{|\RR|, |\BB|\}=11$. Note that $F_2(G)$,  shown in Figure~\ref{token-cont}, does not satisfy Hall's condition for $\AA=\{13, 14, 15, 16, 17, 23\}$, i.e., 
 $|N(\AA)| <|\AA|$. 
 
\begin{figure}[ht] % 
\hfill
\begin{minipage}[t]{.49\textwidth}
\begin{center}
\epsfig{file=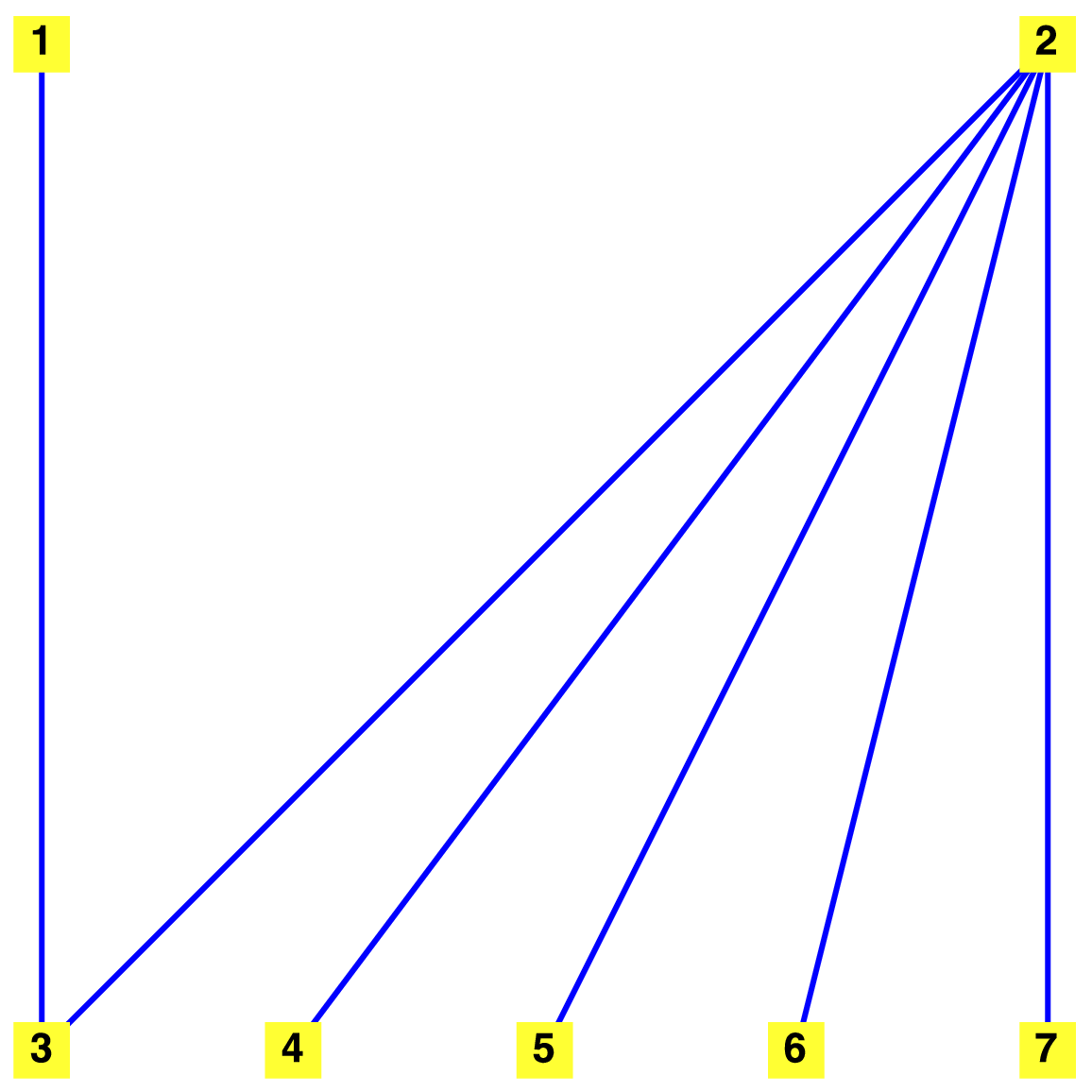, scale=0.27} 
\caption{A bipartite graph $G$ with bipartition $\{B,R\}$, and 
$|B|=2, |R|=5$.}
\label{g-cont}
\end{center}
\end{minipage}
\hfill
\begin{minipage}[t]{.49\textwidth}
\begin{center}
\epsfig{file=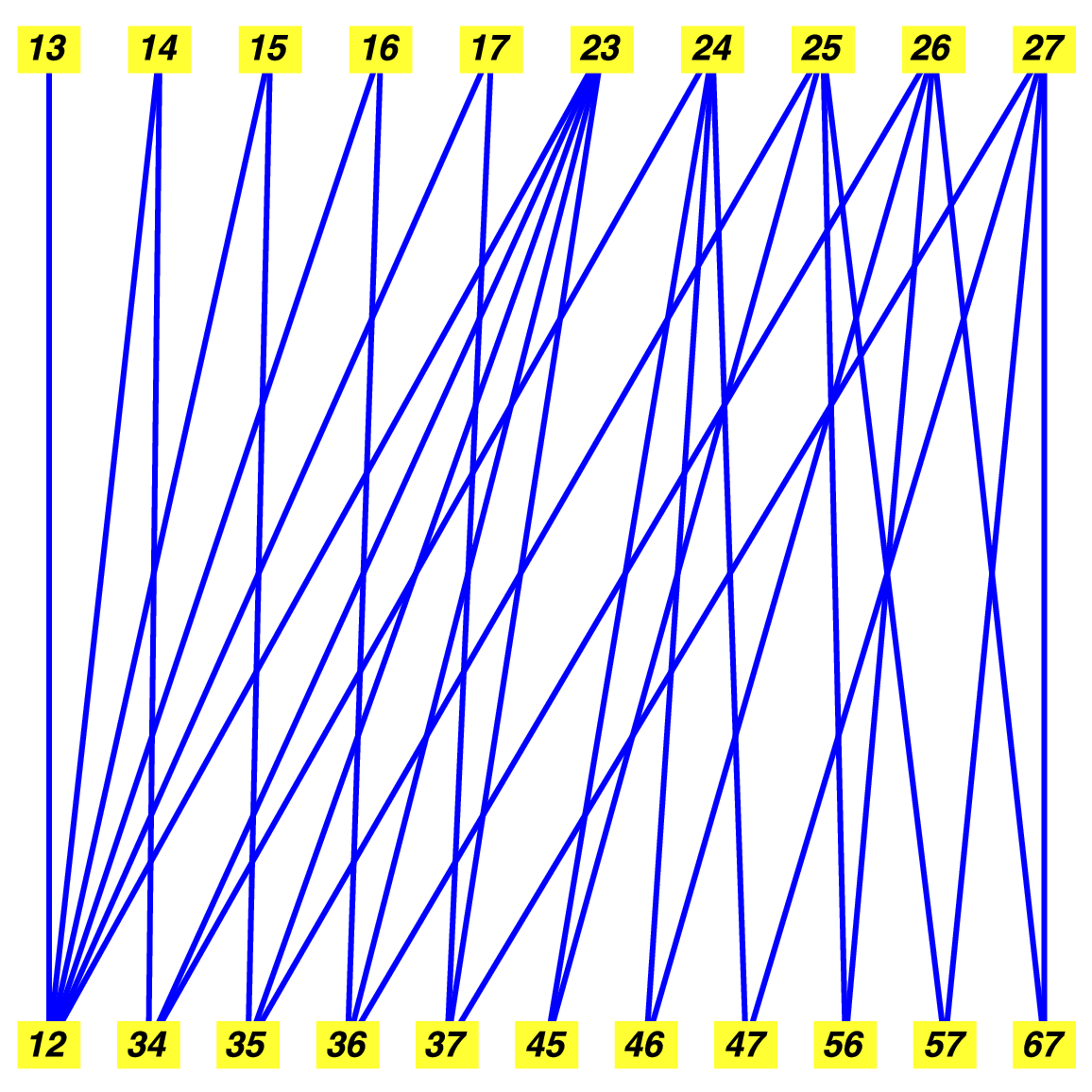, scale=0.27} 
\caption{This is $F_2(G)$ for the graph $G$ on the left. Note that  $\{13, 14, 15, 16, 17, 23, 45, 46, 47, 56, 57, 67\}$  
is an independent set.}
\label{token-cont}
\end{center}
\end{minipage}
\hfill
\end{figure} 
\begin{proposition}\label{prop:sizes}
If $k=2$, then $|\BB|\geq |\RR|$ if and only if $n-m \geq \frac{1+\sqrt{1+8m}}{2}$. 
\end{proposition} 

\begin{proof} From our assumption that $k=2$ and 
Proposition~\ref{prop:bipartite} it follows that $F_2(G)$ has bipartition 
$\{\RR, \BB\}$ with $|\RR|=(m+s)m$ and 
$|\BB|=\binom{2m+s}{2}-(m+s)m$, where $s:=n-m$. Thus $|\BB|-|\RR|= \binom{s}{2}-m$. This equality implies that $|\BB|\geq |\RR|$ if and only if $s^2-s-2m\geq 0$.  
The result follows by solving the last inequality for $s$, and considering that $s\geq 0$.
\end{proof} 

%%%%%%%%%%%%%%%%%%%%%%%%%%%%%%%%
%%%%%%%%%%%%%%%%%%%%%%%%%%%%%%%%
%%%%%%%%%%%%%%%%%%%%%%%%%%%%%%%%
%%%%%%%%%%%%%%%%%%%%%%%%%%%%%%%%
%%%%%%%%%%%%%%%%%%%%%%%%%%%%%%%%
%%%%%%%%%%%%%%%%%%%%%%%%%%%%%%%%

\subsection{Exact values for $\alpha(F_k(G))$ for some bipartite $G$}\label{sec:exact}

Our aim in this subsection is to determine the exact independence 
number of the $k$-token graphs of some common bipartite graphs. 

Next result is a consequence of Theorem~\ref{th:match} (1) (see also \cite{volk}).

\begin{theorem}\label{indep-over-2}
Let $G$ be a bipartite graph. If $G$ has a perfect matching and $k$ is odd, then $\alpha(F_k(G))=\binom{m+n}{k}/2$.
\end{theorem}
\begin{proof}
From Theorem~\ref{th:match} (1) it follows that $F_k(G)$ has a perfect matching. This and the fact that $F_k(G)$ is a bipartite graph imply that $\alpha(F_k(G))=\binom{m+n}{k}/2$.
\end{proof}

\begin{corollary}\label{cor:pn-kmm} For $G\in \{P_{2n},C_{2n},K_{n, n}\}$ and $k$ odd, $\alpha(F_k(G))=\binom{2n}{k}/2$.
\end{corollary}

We noted that for $0\leq m < n$, $T(n, m):=\binom{2n}{2m+1}/2$ is a formula for the sequence A091044 in the ``On-line Encyclopedia of Integer Sequences" (OEIS) \cite{oeis}, and so Corollary \ref{cor:pn-kmm} provides a new interpretation for such a sequence.

As we will see, most of the results in the rest of the section exhibit families of graphs for which the bound for $\alpha(F_k(G))$ 
given in Proposition \ref{prop:bipartite} is attained.

\begin{proposition}\label{prop:k1n} Let $G=K_{m,n}$, with $m=1$ and $n\geq 1$ (i.e., 
 $G$ is the star of order $n+1$). Then

\[
\alpha(F_k(K_{1,n})) = \left\{
  \begin{array}{lr}
  \binom{n}{k} &  k \leq (n+1)/2,\\
    \binom{n}{k-1} &  k > (n+1)/2.
    
  \end{array}
\right.
\]

\end{proposition}

\begin{proof} Let $V(G)=\{v_1, \dots, v_{n+1}\}$ and let $v_1$ be the central vertex of $G$. Since $\alpha(G)=n$, the assertion holds for $k \in \{1, n\}$. So we assume that $2 \leq k \leq n-1$. In this proof we take $R=\{v_1\}$ and $B=\{v_2, \dots, v_{n+1}\}$. Thus, the bipartition 
$\{\mathcal R, \mathcal B\}$ of $F_k(G)$ is given by 
$\mathcal R=\{A \in V(F_k(G)): v_1 \in A\}$ and
 $\mathcal B=V(F_k(G))\setminus \mathcal R$. Thus $|\mathcal R|=\binom{n}{k-1}$ and $|\mathcal B|=\binom{n}{k}$. Note that $F_k(G)$ is biregular: $d(A)=n+1-k$ for every $A \in \mathcal R$ and $d(B)=k$ for every $B \in \mathcal B$.
 
Suppose that $k \leq (n+1)/2$. Then $|\mathcal B|\geq |\mathcal R|$. Now we show
that $|N(\AA)| \geq |\AA|$ for any $\AA\subseteq \mathcal R$.
 Since $N(\AA)=\cup_{A\in \AA}N(A)\subseteq \BB$ and every vertex of $\mathcal B$ has degree $k$, then every vertex in $N(\AA)$ appears at most $k$ times in the disjoint union $\biguplus_{A\in \AA}N(A)$. Therefore
  $|N(\AA)| \geq (n+1-k)|\AA|/k\geq |\AA|$, because $n+1-k\geq k$.
From Hall's Theorem and Lemma~\ref{equal-part} we have 
$\alpha(F_k(G))=|\BB|=\binom{n}{k}$, as desired.

The case $k > (n+1)/2$ can be verified by a totally analogous argument.   
\end{proof}

The number $\alpha(F_2(K_{1,n-1}))$ is equal to $A000217(n-2)$, for $n \geq 4$, where $A000217$ is the sequence of  triangular numbers \cite{oeis}.

\begin{proposition}\label{super-bipartite} 
Let $G, B, R, \mathcal B, \mathcal R$ and $k$ be as in Remark \ref{rem:notation}.
If $\alpha(F_k(G))$ is equal to $\max \{|\RR|, |\BB|\}$ and $G'$ is a bipartite supergraph of $G$ with bipartition $\{R, B\}$, then $\alpha(F_k(G'))= \max \{|\RR|,|\BB|\}$.
\end{proposition}
\begin{proof}
The equality $V(G)=V(G')$ implies $V(F_k(G))=V(F_k(G'))$ and $E(F_k(G))\subseteq E(F_k(G')$. From Proposition~\ref{prop:bipartite} it follows $\alpha(F_k(G')) \geq \max \{|\RR|,|\BB|\}$. On the other hand, since every independent set of $F_k(G')$ is an independent set of $F_k(G)$, we have $\alpha(F_k(G')) \leq \alpha(F_k(G))=\max \{|\RR|,|\BB|\}$.  
\end{proof}

\begin{theorem}\label{super: per-almost}
If $G'$ is a bipartite supergraph of $G$ with bipartition $\{R, B\}$, and $G$ has either a perfect matching or an almost perfect matching, then $\alpha(F_k(G'))=\max\{|\RR|,|\BB|\}$.
\end{theorem}

\begin{proof} In view of Proposition~\ref{super-bipartite}, it is enough 
to show that if $G$ is either a perfect matching or an almost perfect matching, then $\alpha(F_k(G))=\max\{|\RR|,|\BB|\}$.

Suppose that $G$ is a perfect matching. Then $n=m$ and $|G|=2m$. If $k$ is odd, then, by Theorem~\ref{th:match} (1), $F_k(G)$ has a perfect matching. This fact together with Lemma \ref{equal-part} imply $\alpha(F_k(G))=\max\{|\RR|,|\BB|\}$. For $k$ even, Theorem~\ref{th:match} (2) implies: (i) that the set $\SS$ of isolated vertices of $F_k(G)$ has exactly $\binom{m}{k/2}$ elements (see the last paragraph of the proof of Theorem~\ref{th:match} (2)), and (ii) the existence of a matching $M$ of $F_k(G)$ such that $V(M)=V(F_k(G))\setminus \SS$. Now, from the definition of $\RR$ it follows that if $k/2$ is odd, then $\SS \subseteq \RR$, and if $k/2$ is even, then $\SS\subseteq \BB$. Therefore, we have that either $M$ is a matching of $\RR$ into $\BB$ or $M$ is a matching of $\BB$ into $\RR$. In any case, Lemma~\ref{equal-part} implies $\alpha(F_k(G))=\max\{\RR, \BB\}$.  

Now suppose that $G$ is an almost perfect matching. Then $|E(G)|=|B|=m=n-1$ and $G$ has exactly one isolated vertex in $R$, say $u$. 
From Theorem~\ref{th:match} (2) it follows: (i) that the set $\SS$ of isolated vertices of $F_k(G)$ has exactly $\binom{m}{\floor{ k/2}}$ elements (see the last paragraph of the proof of Theorem~\ref{th:match} (2)), and (ii) the existence of a matching $M$ of $F_k(G)$ such that $V(M)=V(F_k(G))\setminus \SS$. Again, it is easy to see that either $\SS \subseteq \RR$ or $\SS\subseteq \BB$. Then either $M$ is a matching of $\RR$ into $\BB$ or $M$ is a matching of $\BB$ into $\RR$. In any case, Lemma~\ref{equal-part} implies $\alpha(F_k(G))=\max\{|\RR|, |\BB|\}$.   
\end{proof}

Our next result is an immediate consequence of Proposition~\ref{prop:bipartite} and Theorem~\ref{super: per-almost}.

\begin{corollary}\label{k2-path-kmm}
Let $t$ be a positive integer. If $G\in \{P_{t}, K_{t,t}, K_{t, t+1}\} $ and $k$ is an integer such that $1\leq k \leq |G|-1$,  
then \[\alpha(F_k(G))=\max \{ r,    \binom{p}{k}-r \},\]
where $p:=|G|$ and $r:=\sum_{i=1}^{\lceil k/2\rceil} \binom{\lceil p/2 \rceil}{2i-1}\binom{\lfloor p/2 \rfloor}{k-2i+1}$.
\end{corollary}

It is a routine exercise to check that $\alpha(F_2(P_m))=\lfloor m^2/4 \rfloor$ and that sequence $\{\alpha(F_2(P_m))\}_{m\geq 0}$ coincides with A002620 in OEIS~\cite{oeis}.

The following conjecture has been motivated by the results of 
Corollary~\ref{k2-path-kmm} for  $\alpha(K_{t, t})$ and
 $\alpha(K_{t, t+1})$, and experimental results.  Our aim in the next section is to show Conjecture~\ref{conj-complete-bi} for $k=2$.  
 
\begin{conjecture}\label{conj-complete-bi}
If $G$ is a complete bipartite graph with partition $\{R, B\}$, then $\alpha(F_k(G))=\max \{|\RR|, |\BB|\}$.  
\end{conjecture}

%%%%%%%%%%%%%%%%%%%%%%%%%%%%
%%%%%%%%%%%%%%%%%%%%%%%%%%%%
%%%%%%%%%%%%%%%%%%%%%%%%%%%%
%%%%%%%%%%%%%%%%%%%%%%%%%%%%
%%%%%%%%%%%%%%%%%%%%%%%%%%%%
%%%%%%%%%%%%%%%%%%%%%%%%%%%%
%%%%%%%%%%%%%%%%%%%%%%%%%%%%

\section{Proof of Theorem \ref{th:Kmn}}\label{sec:Kmn}
 
As usual, for a nonnegative integer $t$, we use $[t]$ to denote the set 
$\{1,\ldots ,t\}$, and for a finite set $X$, $C^{X}_2$ to denote the set of all $2$-sets of $X$.

Throughout this section, $B,R,\BB, \RR,m$ and $n$ are as in Remark~\ref{rem:notation} for $G=K_{m,n}$ and $k=2$.  

Let $s:=n-m$ (we recall that $m\leq n$), $s_0:=\frac{1+\sqrt{1+8m}}{2}$, $B=\{b_1,\ldots ,b_m\}$, and 
 $R=\{r_1,\ldots ,r_m, \ldots ,r_{m+s}\}$. 
   
Note that Proposition~\ref{prop:k1n} implies Theorem \ref{th:Kmn} when
 $m=1$. Thus, we may assume that $m \geq 2$, and hence that $s_0 >2$. Similarly, 
 because Corollary~\ref{k2-path-kmm} implies Theorem  \ref{th:Kmn} when $s \in \{0, 1\}$, we also assume that $s\geq 2$.

As we have seen in Proposition~\ref{prop:sizes},
 for $k=2$ the value of $\max\{|\BB|, |\RR|\}$ depends on the value of 
 $\max\{s, s_0\}$. Depending on whether $s< s_0$ or $s_0 \geq s$ we use
  certain subgraphs $G_1$ and $G_2$ of $K_{m,n}$, which, as we will see, satisfy that 
  $\alpha(F_2(K_{m,n}))=\alpha(F_2(G_1))=|\RR|$ if $|\RR|>|\BB|$, and
  $\alpha(F_2(K_{m,n}))=\alpha(F_2(G_2))=|\BB|$ otherwise.
 
For $2\leq s < s_0$ consider the subgraph $G_1$ of $K_{m,n}$ defined as follows: since $2\leq s <s_0$ implies that $m > \binom{s}{2}$, we can take an injective function, say $\phi$, from $C^{[s]}_2$ to $[m]$. Let $V(G_1)=B\cup R$ and  
\[
E(G_1)=\{[b_i,r_i]: i \in [m]\} \bigcup \{ [b_{\phi(\{i,j\})}, r_{m+l}]: \{i,j\}\in C^{[s]}_2 \mbox{ and } l\in\{i,j\}\}.
\]

\begin{lemma}\label{k2bipartite1} If $G_1$ is as above, then 
$\alpha(F_2(G_1))=|\RR|=\max\{|\BB|, |\RR|\}$.
\end{lemma}

\begin{proof}
From $s<s_0$ and Proposition \ref{prop:sizes} we have that $|\RR|>|\BB|$.  Thus, by Lemma~\ref{equal-part} and Hall's Theorem, it is enough to show that for every $X\subseteq \BB$, $|N(X)|\geq |X|$. 

Let $X\subseteq \BB$. From the definition of $\BB$ and $k=2$ we know that $X$ is a collection of pairs of vertices in $B\cup R$ satisfying that each such pair have both elements in $B$ or both in $R$. Let $X_1$ be the set of pairs in $X$ with both elements in $B$, let $X_2$ be the set of pairs in $X$ which have at least one element in $\{r_1,\ldots, r_m\}$, and let 
$X_3$ be the set of pairs in $X$ which have both elements in $\{r_{m+1},\ldots, r_n\}$. Clearly, $\{X_1, X_2, X_3\}$ is a partition of $X$. 

In the rest of the proof, if $\{x_i,x_j\}\in X_1\cup X_2\cup X_3$, then we shall assume that $i<j$.  Let us define
 \begin{eqnarray*}
 X_1'&:=&\left\{\{r_i, b_j\} \colon \{b_i, b_j\} \in X_1\right\},\\
  X_2'&:=&\left\{\{b_i, r_j\}\colon \{r_i, r_j\} \in X_2\right\},\\
  X_3'&:=&\left\{\{b_{\phi(\{i, j\})}, r_{m+j}\} \colon \{r_{m+i}, r_{m+j}\} \in X_3\right\}.
  \end{eqnarray*}

%%%%%%%%%%%%%%
%%%%%%%%%%%%%%
%%%%%%%%%%%%%%

 Note that $X_1'\cap X_2'=\emptyset$ and $X_1'\cap X_3'=\emptyset$.
  From the definition of $E(G_1)$ it follows that $X_l' \subseteq N(X_l)$ for every $l\in \{1,2,3\}$. Hence 
  $X_1'\cup X_2'\cup X_3' \subseteq N(X)$.  Also note that $|X_l'|= |X_l|$ for every $l$ (for $l=3$ take into account that $\phi$ is injective). Notice that if $s\geq 2$  and   $X_2'\cap X_3' = \emptyset$, then 
 $|N(X)| \geq |X_1' \cup X_2' \cup X'_3| \geq |X|$, as required. Thus, we can assume that $s\geq 2$ and 
 $X_2'\cap X_3' \neq \emptyset$. Let
   \[
   X'_{2,3}:=\left\{\{b_{\phi(\{i, j\})}, r_{\phi(\{i, j\})}\} \colon \{b_{\phi(\{i, j\})}, r_{m+j}\} \in X_2'\cap X_3'\right\}.
   \] 
   First we show that $X'_{2,3} \subseteq N(X_2)$. Let $\{b_{\phi(\{i, j\})}, r_{\phi(\{i, j\})}\} \in X'_{2,3}$, then  $\{b_{\phi(\{i, j\})}, r_{m+j}\}$ belongs to $X_2'$ and  $\{r_{\phi(\{i, j\})}, r_{m+j}\} \in X_2$ by definition of $X_2'$. The result follows because
   \[
   \left[\{b_{\phi(\{i, j\})}, r_{\phi(\{i, j\})}\}, \{r_{\phi(\{i, j\})}, r_{m+j}\}\right] \in E\left(F_2(G_1)\right).
   \] 
   
    It is clear that $X'_{2,3} \cap X_l'=\emptyset$ for every $l$. Since $\phi$ is injective, the equality $\{b_{\phi(\{i, j\})}, r_{\phi(\{i, j\})}\}=\{b_{\phi(\{l, t\})}, r_{\phi(\{l, t\})}\}$, with $i<j$ and $l<t$, implies that $i=l$ and $j=t$, and hence $|X'_{2,3}|= |X_2'\cap X_3'|$.  Therefore, by the inclusion-exclusion principle we have   
   \[
  |N(X)|\geq  |X_1'\cup X_2'\cup X_3' \cup X_{2,3}|=|X_1'|+|X_2'|+|X_3'|+|X'_{2,3}| -|X_2'\cap X_3'| = |X|.
   \]
\end{proof}
 
For $s \geq s_0$ consider the subgraph $G_2$ of $G$ defined as follows: since
 $s \geq s_0$ implies that $m \leq \binom{s}{2}$, we can take an injective function, say $\phi$, from $[m]$ to $C^{[s]}_2$. 
Let $V(G_2)=B\cup R$ and
$$E(G_2) = \{[b_i, r_i]: i \in [m]\} \bigcup 
 \{[b_i,r_{m+l}]: i\in [m], l\in\{i_1,i_2\} \mbox{ where } \{i_1, i_2\}=\phi(i)\}.$$
 
 %%%%%%%%%%%%%%%%%%%%%%%%%%
 %%%%%%%%%%%%%%%%%%%%%%%%%%
 %%%%%%%%%%%%%%%%%%%%%%%%%%
 %%%%%%%%%%%%%%%%%%%%%%%%%%
\begin{lemma}\label{k2bipartite2}
If $G_2$ is as above, then $\alpha(F_2(G_2))=|\BB|=\max\{|\BB|, |\RR|\}$.
\end{lemma}

\begin{proof} From $s \geq s_0$ and Proposition \ref{prop:sizes} we have that 
$|\BB|\geq |\RR|$. Thus, by Lemma~\ref{equal-part} and Hall's Theorem, it is enough to show that 
for every $X\subseteq \RR$, $|N(X)|\geq |X|$. 

Let $X\subseteq \RR$. From the definition of $\RR$ and $k=2$ we know that $X$ is a collection of pairs of vertices in $B\cup R$ such that each pair is formed by a vertex in $B$ and the other one in $R$. Thus, $X$ corresponds naturally to a subset of edges of $K_{m,n}$.   

Without loss of generality, we assume that $B\cup R$ are points in the plane. 
 More precisely, for $i\in [m]$ and $j\in [n]$, we assume that $b_i$ and $r_j$ are the points with coordinates $(i,1)$ and
  $(j,0)$, respectively. Thus we shall think of the elements in $X$ as straight edges joining a vertex in $B$ to a vertex in $R$.  

 Let $X_1$ and $X_2$ be the set of edges in $X$ with positive and negative slope, respectively, and let $X_3$ be the vertical
edges in $X$. Clearly, $\{X_1, X_2, X_3\}$ is a partition of $X$. Let us define
\begin{eqnarray*}
X_1'&:=&\left\{\{b_j, b_i\} \colon [r_j, b_i] \in X_1\right\},\\
X_2'&:=&\left\{\{r_i, r_j\} \colon [b_i, r_j]\in X_2\right\},\\
X_3'&:=&\left\{\{r_{m+i_1}, r_{m+i_{2}}\} \colon \{b_{i}, r_{i}\} \in X_3 \mbox{ and } \phi(i)=\{i_1,i_2\} \right\}.
\end{eqnarray*}
From the definition of $E(G_2)$ it follows that $X_1'\cup X_2'\cup X_3' \subseteq N(X)\subseteq \BB$. Note that $X_1',X_2',$ and $X_3'$ are pairwise disjoint.
Since  $|X_l'|=|X_l|$ for $l\in \{1,2,3\}$ (because $\phi$ is injective), then, by the inclusion-exclusion principle:
\[
|N(X)| \geq |X_1'|+|X_2'|+|X_3'|= |X_1|+|X_2|+|X_3|= |X|.
\]
\end{proof}
\noindent{\em Proof of Theorem \ref{th:Kmn}}. Let $G_1$ and $G_2$ be as above. Clearly, exactly one of $G_1$ or $G_2$ exists and is a subgraph of $K_{m,n}$. Let $H$ be
such a subgraph. From Lemmas~\ref{k2bipartite1} and~\ref{k2bipartite2} we have 
$\alpha(F_2(H))=\max \{|\RR|, |\BB|\}$. This and Proposition~\ref{super-bipartite} imply $\alpha(F_2(H))=\alpha(F_2(K_{m,n}))$. 
$\square$

%%%%%%%%%%%%%%%%%%%%%%%%%%%%
%%%%%%%%%%%%%%%%%%%%%%%%%%%%
%%%%%%%%%%%%%%%%%%%%%%%%%%%%
%%%%%%%%%%%%%%%%%%%%%%%%%%%%
%%%%%%%%%%%%%%%%%%%%%%%%%%%%
%%%%%%%%%%%%%%%%%%%%%%%%%%%%
%%%%%%%%%%%%%%%%%%%%%%%%%%%%
%%% cycle graphs %%%%%%%%%%%%%%%%%%
%%%%%%%%%%%%%%%%%%%%%%%%%%%%

\section{Proof of Theorem~\ref{k2cicle}}\label{sec:C_p}
 
Again, we first need to give some preliminary results and notation. We start by stating recursive inequalities for $\alpha(F_k(G))$.  

\begin{lemma}\label{th:independent}
Let $G$ be a graph of order $n$. For $2\le k\le n-1$, we have
{\small \begin{equation}\label{eq:independent}
    \displaystyle\max_{v\in V(G)} \left\{\alpha\left(F_{k-1}(G-v)\right) + \alpha\left(F_k\left(G-N[v]\right)\right)\right\} \le \alpha\big(F_k(G)\big) \le \frac1{k}\,\displaystyle\sum_{v\in V(G)} \alpha\big(F_{k-1}(G-v)\big).
\end{equation}
}
\end{lemma}

\begin{proof}
We begin by proving the right  inequality of (\ref{eq:independent}). Let $\II$ be an independent set of vertices in $F_k(G)$ with maximum cardinality. For $v\in V(G)$,
let $\II_v$ be the set formed by all the elements of $\II$ containing $v$. Since every vertex of $\II$ is a $k$-set of $V(G)$, then 
$k | \II| =\sum_{v\in V(G)}  |\II_v|$. Furthermore, note that the collection  $\{A\setminus\{v\} \colon~  A\in \II_v \}$ is an independent set of $F_{k-1}(G-v)$,
 and so $|\II_v | \le \alpha\big(F_{k-1}(G-v)\big)$ for every $v\in V(G)$. The desired inequality follows from previous relations and the fact that $\alpha\big(F_k(G)\big)= |\II|$. 

Now we show the left inequality. For $v\in V(G)$, let $\II_{\neg v}$ 
(respectively $\mathcal J_{\neg v}$) be an independent set in $F_{k-1}(G-v)$ 
(respectively $F_k(G-N[v])$) with maximum cardinality. Then 
$|\II_{\neg v}|=\alpha(F_{k-1}(G-v))$ and $|\mathcal J_{\neg v}|=\alpha(F_k(G-N[v]))$.
Let $\II_v$ be the collection of sets $\left\{A\cup\{v\} \colon~A\in \II_{\neg v}\right\}$. From the construction of $I_v$ and $J_{\neg v}$ it is easy to see that $\mathcal I_v\cap \mathcal J_{\neg v} =\emptyset$, and that $\mathcal I_v\cup \mathcal J_{\neg v}$ form an independent set of $F_k(G)$. Since the last two statements hold for every $v\in V(G)$, the required inequality follows.  
\end{proof}

\begin{remark}
The bounds for $\alpha(F_k(G))$ in (\ref{eq:independent}) are best possible: for instance, the left (respectively right) hand side of (\ref{eq:independent}) is reached when $G\simeq K_{1,3}$ and $k=2$ (respectively $G\simeq K_n$ and $k=2$). 
\end{remark}

We recall that a graph $H$ is vertex-transitive if given any two vertices $u$ and $v$ in $V(H)$, there is an automorphism of $H$ mapping $u$ to $v$.

\begin{corollary}\label{cor:indep}
Let $G$ be a vertex-transitive graph of order $n$ and let $w$ be any vertex in $G$. For $2\le k\le n-2$, we have
\[ \alpha\big(F_k(G)\big) \le \min \left\{ \frac{n}{k}\, \alpha\big(F_{k-1}(G-w)\big), \frac{n}{n-k}\, \alpha\big(F_{k}(G-w)\big) \right\}.\]
\end{corollary}

\begin{proof} 
 Since $G$ is vertex-transitive, then $\alpha(F_{k-1}(G-w))=\alpha(F_{k-1}(G-u))$ 
for any $u\in V(G)$. From this and Theorem \ref{th:independent} it follows that
 
$$\alpha(F_k(G)) \le \frac{n}{k}  \alpha (F_{k-1}(G-w)).$$

In a similar way we can deduce that 

$$\alpha(F_{n-k}(G)) \le \frac{n}{n-k}  \alpha (F_{n-k-1}(G-w)).$$

The desired inequality follows from the previous inequality and considering that  $ F_{k}(G) \simeq F_{n-k}(G)$, and that $F_{k}(G-w)\simeq F_{(n-1)-k}(G-w)$.
\end{proof}

Applying Lemma \ref{th:independent} and Corollary \ref{cor:indep} to $G\simeq C_n$ and $G\simeq K_n$, we have the following corollary (we remark that equation (\ref{eq:IndKn}) is in fact a theorem of Johnson \cite{Jo}):

\begin{corollary}\label{p:indep}
For $2\le k\le n-2$ we have
{\footnotesize \begin{equation}\label{eq:IndCn}
    \alpha\big(F_{k-1}(P_{n-1})\big) + \alpha\big(F_k(P_{n-3})\big) \le \alpha\big(F_k(C_n)\big) \le \min \left\{ \frac{n}{k}\, \alpha\big(F_{k-1}(P_{n-1})\big), \frac{n}{n-k}\, \alpha\big(F_{k}(P_{n-1})\big) \right\}
\end{equation}
}
and
{\footnotesize 
\begin{equation}\label{eq:IndKn}
    \alpha\big(J(n-1,k-1)\big) \le \alpha\big(J(n,k)\big) \le \min \left\{ \frac{n}{k}\, \alpha\big(J(n-1,k-1)\big), \frac{n}{n-k}\, \alpha\big(J(n-1,k)\big) \right\}.
\end{equation}
}
\end{corollary}

As $F_2(C_{3})\simeq C_3$, then $\alpha(F_2(C_{3}))=1$.  Thus for the rest of this section we assume that $p\geq 4$. 

Let $V(C_p):= \{1, \dots, p\}$ and 
$E(C_p):=\{[i, {i+1}]\colon i=1, \ldots , p-1\}\cup \{[p, {1}]\}.$
If $X, Y \subseteq V\left(F_2(C_{p})\right)$,  we say that $X$ and $Y$ are  {\it linked} in $F_2(C_{p})$ if and only if $F_2(C_{p})$ contains an edge $[A, B]$ such that $A \in X$ and $B\in Y$. We use $X\approx Y$ to denote that $X$ and $Y$ are linked in $F_2(C_{p})$. If $A, B \in V\left(F_2(C_{p})\right)$, we use $A\triangle B$ to denote the symmetric difference between $A$ and $B$. Recall that for $A, B \in V\left(F_2(C_{p})\right)$,  $[A, B]\in E(F_2(C_p))$ if and only if either $A \triangle B=\{t, {t+1}\}$ for 
$1\leq t \leq p-1$, or $A \triangle B=\{1, p\}$.  

For $i=1, \ldots , p-1$, let  $L_i:=\left\{\{j, p-(i-j)\right\}\colon 1\leq j \leq i\} \subseteq V(F_2(C_p))$ (see Figure~\ref{t2-cycle}). 
Each assertion in the following observation follows easily from the definition of $L_i$.  

\begin{figure}[ht]
  \begin{center}
   \includegraphics[width=4.5in]{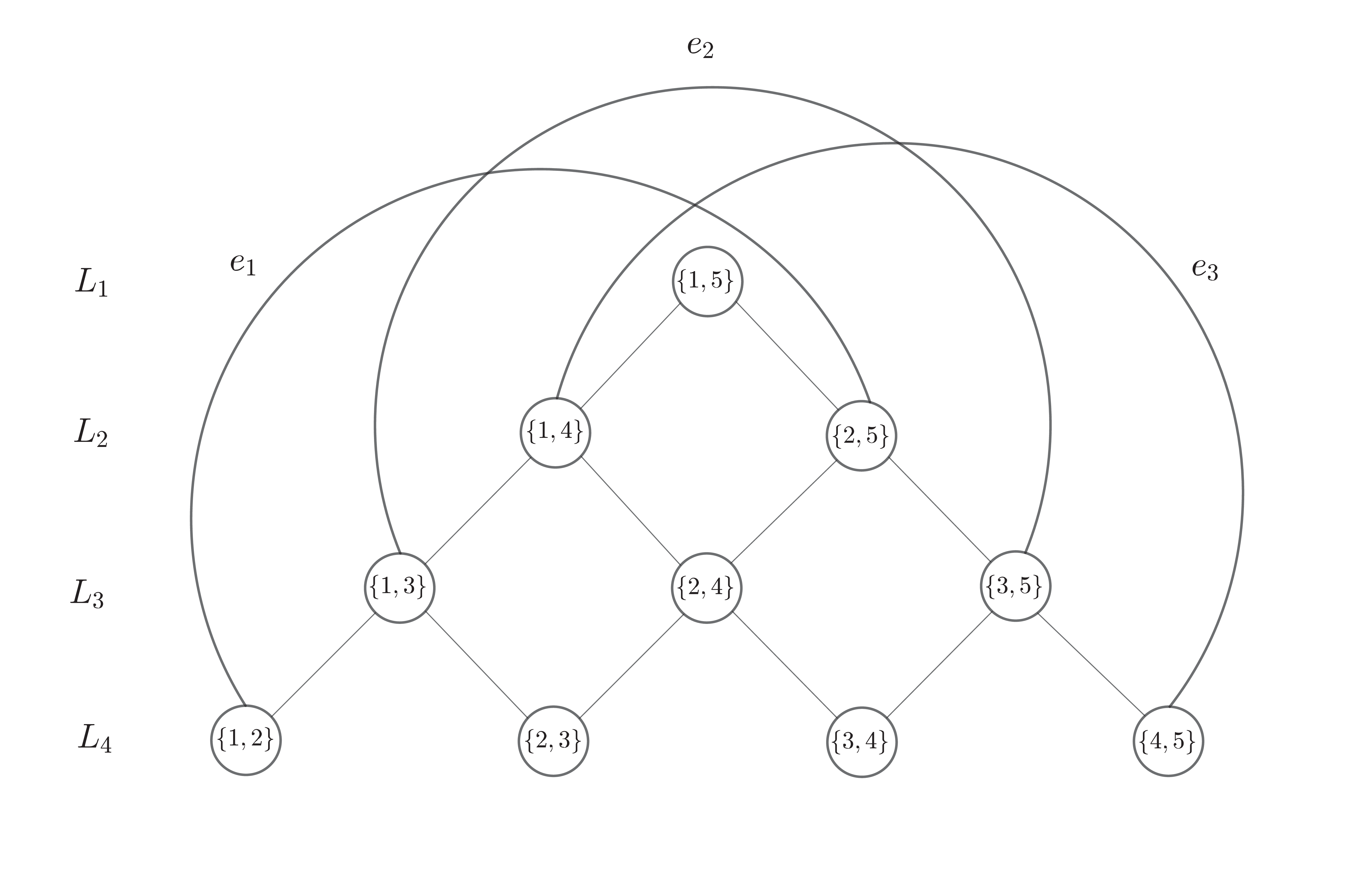}
  \end{center}
\caption{\small  Here is shown $F_2(C_5)$. Note that $F_2(P_5)\simeq F_2(C_5)\setminus\{e_1, e_2, e_3\}$ and that $L_1=\{\{1,5\}\}, \dots,  L_4=\{\{1, 2\}, \{2, 3\}, \{3, 4\}, \{4, 5\}\}$.}
 \label{t2-cycle}
\end{figure}

\begin{observation}\label{obs-cycle}  For $p$ and $L_i$ as above, we have that:
\begin{enumerate}
\item If  $i\in \{1,\ldots ,p-2\}$, then $L_i\approx L_{i+1}$ and $[\{i, p\},\{i+1, p\}]$ is an edge of $F_2(C_p)$ witnessing this fact.
\item If $i\in \{2,\ldots ,p-1\}$, then $L_i\approx L_{p-i+1}$ and $[\{1, i\},\{i, p\}]$ is an edge of $F_2(C_p)$ witnessing this fact.
\item If $L_i\approx L_j$ for some $i,j \in \{1,\ldots ,p-1\}$ with $i\neq j$, then $L_i\approx L_j$ is one of the described in (1) or (2).

\end{enumerate}
\end{observation}

\noindent {\em Proof of Theorem~\ref{k2cicle}}. 
First, we show that $\alpha(F_2(C_p)) \leq \floor{p \floor{p/2}/2}$. 
 
By Corollary~\ref{p:indep}  and Corollary~\ref{k2-path-kmm}, we have that
\[
\alpha(F_2(C_p)) \leq \min \left \{p/2{\lceil}(p-1)/2{\rceil}, p/(p-2)\floor{(p-1)^2/4}\right\}.
\]
If $p=2t$, then $p/2{\lceil}(p-1)/2{\rceil}=p/(p-2)\floor{(p-1)^2/4}$ and hence 
$\alpha(F_2(C_p)) \leq p^2/4$. As $p$ is even, $p^2/4=\floor{p\floor{p/2}/2}$, as desired.

Now, if $p=2t+1$, we have
\[
\alpha(F_2(C_p)) \leq (p/2){\lceil}(p-1)/2{\rceil}=pt/2.
\]
Thus $\alpha(F_2(C_p)) \leq \lfloor p t / 2 \rfloor =  \floor{p \floor{p/2}/2}$, because $t= \lfloor p / 2 \rfloor$.

Now we show that $\alpha(F_2(C_p)) \geq \floor{p \floor{p/2}/2}$. 

Let $i \in \{1, \dots, p-1\}$, and let $t=\lfloor p/2 \rfloor$. 
Note that if $p=2t$, then $L_i$ is an independent set of $F_2(C_p)$. For $p=2t+1$ we have that $L_i$ is also an independent except when 
$i = t+1$. 

%%%%%%%%%%%%%%%%%%%%%%%%%%%%%%%%%%%%%%%%%%%%%%%%
%%%%%%%%%%%%%%%%%%%%%%%%%%%%%%%%%%%%%%%%%%%%%%%%

{\it Case 1.} Suppose that $p=2t$. From previous paragraph and Observation~\ref{obs-cycle} we have that 
 $I=L_1 \cup L_3 \cup  \dots \cup L_{p-1}$ is an independent set of $F_2(C_p)$. 
 Since $|I|=p^2/4$, and $p^2/4=\floor{p\floor{p/2}/2}$ because $p$ is even, then we are done.\\

{\it Case 2.} Suppose that $p=2t+1$. We split the rest of the proof depending on whether $t$ is odd or even.

{\it Case 2.1.}  $t$ is odd. From Observation \ref{obs-cycle} it follows that $$\{L_1, L_3, \dots ,L_t, L_{t+3}, L_{t+5}, \dots, L_{p-1}\}$$ is a collection of pairwise non-linked independent sets. Then, 
\[
 I=L_1 \cup L_3 \cup \dots \cup L_t \cup L_{t+3} \cup L_{t+5} \cup \dots \cup L_{p-1}
\]  is an independent set in $F_2(C_p)$. But 
\[
|I|=\left(1+3+\dots +t\right)+\left((t+3)+ (t+5) + \dots + (p-1)\right)=\frac{1}{2}(tp-1)=\floor{p \floor{p/2}/2}.
\]

{\it Case 2.2.}  $t$ is even. Similarly, $\{L_1, L_3, \dots, L_{t-1}, L_{t+2}, L_{t+4}, \dots, L_{p-1}\}$
  is a collection of pairwise non-linked independent sets, and hence
\[
I=L_1 \cup L_3 \cup \dots \cup L_{t-1} \cup L_{t+2} \cup L_{t+4} \cup \dots \cup L_{p-1}
\]
 is an independent set in $F_2(C_p)$. In this case we have that
 \[
 |I|=\left(1+3+ \dots + t-1\right)+\left((t+2)+(t+4)+\dots +p-1\right)=\frac{1}{2}tp=\floor{p \floor{p/2}/2}. \mbox{   } \square \]

\section*{Acknowledgments}
J. L. was partially supported by CONACyT Mexico
grant 179867. L. M. R. was partially supported by PROFOCIE (2015-2018) grant, trough UAZ-CA-169.

\end{document}